\documentclass[a4paper,11pt]{article}

\usepackage[top=3.0cm,bottom=3.0cm,left=2.5cm,right=2.5cm]{geometry}
\usepackage{tikz}
\usepackage{graphics,epsfig,psfrag}
\usepackage{graphicx}
\usepackage{amsfonts}
\usepackage{amssymb}
\usepackage{amsmath}
\usepackage{amsthm}
\usepackage[english]{babel}
\usepackage{ctex}
\usepackage{hyperref}
\usepackage{cases}
\usepackage{authblk}
\usepackage{mathtools}
\usepackage{pifont}
\usepackage[T1]{fontenc}
\usepackage{enumitem}

\newtheorem{theorem}{Theorem}[section]
\newtheorem{lemma}[theorem]{Lemma}

\newtheorem{claim}[]{\noindent Claim}[]

\newtheorem{observation}[theorem]{Observation}

\begin{document}

\title{Tree-partitions of graphs with bounded tree-depth}

        \author[1]{Rong Chen\footnote{Email: rongchen@fzu.edu.cn.}}
		\author[1]{Huayue Liu\footnote{Email: liuhuayue0625@126.com(corresponding author).}}

\affil[1]{Center for Discrete Mathematics, Fuzhou University\\

Fuzhou, People's Republic of China}

	\date{}
	\maketitle

\begin{abstract}
A tree-partition of a graph $G$ is a partition of $V(G)$ such that identifying the vertices in each part gives a tree. It is known that every graph with treewidth $k$ and maximum degree $\Delta(G)$ has a tree-partition in which each part has size $O(k\cdot\Delta(G))$. In this paper, we study tree-partitions of graphs with bounded tree-depth, a stronger parameter than treewidth, with an additional property for the tree $T$ indexing the tree-partition. We prove that every connected graph with tree-depth $h$ admits a $T$-partition for some tree $T$ with $\operatorname{rad}(T)\leq h-1$ and with width at most $\mathrm{max}\{1, (4h-10)\Delta(G)+1\}$.
\end{abstract}

	

\section{Introduction}
All graphs in this paper are finite and simple.
Treewidth and pathwidth are graph parameters that respectively measure how similar a given graph is to a tree or a path. They are of fundamental importance in structural and algorithmic graph theory. Throughout this paper, let $tw(G)$ and $pw(G)$ denote the treewidth and pathwidth of a graph $G$, respectively.


Let $T$ be a non-empty tree. A \emph{$T$-partition} of a graph $G$ is a partition $(B_x:x\in V(T))$ of $V(G)$ such that, for every edge $uv\in E(G)$ with $u\in B_x$ and $v\in B_y$, either $x=y$ or $xy\in E(T)$. The sets $B_x$ are called the \emph{bags} of the $T$-partition.
The \emph{width} of a $T$-partition is $\max_{x\in V(T)} |B_x|$.
A \emph{tree-partition} of $G$ is a $T$-partition for some tree $T$.
The \emph{tree-partition-width} of $G$, denoted by $tpw(G)$, is the minimum width of a tree-partition of $G$. 
Tree-partitions were introduced by Seese~\cite{Seese1985} and Halin~\cite{Halin1991}, and have since been widely studied. Applications include nonrepetitive graph colouring~\cite{Wood2008}, clustered graph
colouring~\cite{Alon2003,Liu2018}, size-Ramsey numbers~\cite{Draganic2026,Kamcev2021}, and the Erd\H{o}s--P\'osa property~\cite{Chatzidimitriou2018,Raymond2017}.

Bounded tree-partition-width implies bounded treewidth. In particular, for every graph $G$, $tw(G)\le 2~tpw(G)-1$.
On the other hand, the referee of \cite{Ding1995} showed that if the maximum degree and treewidth are both bounded, then so is the tree-partition-width, which is one of the most useful results about tree-partitions.

\begin{theorem}[Ding and Oporowski~\cite{Ding1995}]
\label{thm:treewidth}
For any non-trivial graph $G$, $$tpw(G)\le 24\bigl(tw(G)+1\bigr)\Delta(G).$$
\end{theorem}

Wood~\cite{Wood2009} showed that Theorem~\ref{thm:treewidth} is best possible up to the multiplicative constant. The constant in the upper bound has been steadily improved~\cite{Ding1995,Wood2009}, most recently to $8$ by Distel et al.~\cite{Distel2026}.

Distel and Wood~\cite{DistelWood2024} proved a result similar to Theorem~\ref{thm:treewidth}, with two additional properties for the tree $T$ indexing the tree-partition.

\begin{theorem}[Distel and Wood~\cite{DistelWood2024}]
For any positive integers $k$ and $d$, every non-trivial graph $G$ with $\operatorname{tw}(G)\le k-1$ and $\Delta(G)\le d$ admits a $T$-partition of width at most $18kd$, for some tree $T$ with $\Delta(T)\le 6d$ and
$|V(T)|\le \max\{\frac{|V(G)|}{2k},1\}$.
\end{theorem}



In this paper, we study tree-partitions of graphs with bounded tree-depth, a stronger parameter than treewidth. We show that such graphs admit tree-partitions of bounded width whose indexing trees have the additional property of bounded radius.

To state our result, we need one more definition. For a connected graph $G$, the {\it radius} of $G$, denoted by $\operatorname{rad}(G)$, is equal to 
$\min\limits_{v\in V(G)}\max\limits_{u\in V(G)} d_G(u,v)$, where $d_G(u,v)$ denotes the distance between $u$ and $v$ in $G$.

\begin{theorem}\label{thm:main}
For any connected graph $G$ with tree-depth $h$, there is a tree $T$ with $\operatorname{rad}(T)\leq h-1$ and a
$T$-partition of $G$ with width at most $\mathrm{max}\{1, (4h-10)\Delta(G)+1\}$.
\end{theorem}
Tree-depth alone does not bound tree-partition-width, even for graphs of tree-depth $3$, as shown by the lemma at the end of the paper.

\section{Preliminaries}\label{sec:preliminaries}
For a graph $G$, let $V(G)$ and $E(G)$ denote its vertex and edge sets, respectively. A graph is \emph{non-trivial} if it has at least one edge. For a vertex $v$ in $V(G)$, its \emph{neighbourhood} in $G$ is $N_G(v):=\{w\in V(G): vw\in E(G)\}$. For a set $X\subseteq V(G)$, its \emph{neighbourhood} in $G$ is $N_G(X):=\bigcup\limits_{v\in X}N_G(v)\setminus X$.
Let $d_G(v):=|N_G(v)|$ denote the {\it degree} of $v$ in $G$ and let $\Delta(G)$ denote the \emph{maximum degree} of $G$.

A \emph{tree-decomposition} of a graph $G$ is a pair $(T,\mathcal B)$, where $T$ is a tree and
$\mathcal B=(B_t:t\in V(T))$ is a family of subsets of $V(G)$ such that:
\begin{itemize}
  \item for every edge $uv\in E(G)$, some $t\in V(T)$ satisfies
        $\{u,v\}\subseteq B_t$; and
  \item for every vertex $v\in V(G)$, the set
        $\{t\in V(T):v\in B_t\}$ induces a non-empty subtree
        of $T$.
\end{itemize}
Its width is $\max_{t\in V(T)}|B_t|-1$. The \emph{treewidth} of $G$, denoted by $tw(G)$, is the minimum width of a tree-decomposition of $G$. If the decomposition tree $T$ is a path, the corresponding parameter is the \emph{pathwidth} of $G$, denoted by $pw(G)$.

Let $F$ be a rooted tree with root $r$. For any $s,t\in V(F)$, the vertex $t$ is an \emph{ancestor} of $s$ (and $s$ is a \emph{descendant} of $t$) in $F$ if $t$ lies on the path from $r$ to $s$. Moreover, if  $st\in E(F)$, then $t$ is the \emph{parent} of $s$ and $s$ is a \emph{child} of $t$. 
An ancestor or descendant of a vertex is \emph{strict} if it is distinct from the vertex itself.
Let $Anc_F(t)$ denote the set of strict ancestors of $t$ in $F$, and let $F_t$ denote the subtree of $F$ induced by $t$ and all its descendants. The \emph{depth} of $t$ in $F$ is the number of vertices on the path from $r$ to $t$. The \emph{height} of $F$ is the maximum depth of a vertex in $F$.
The \emph{closure} of $F$, denoted by $clos(F)$, is the graph on $V(F)$ in which two vertices are adjacent whenever one is a strict ancestor of the other in $F$. The \emph{tree-depth} of a graph $G$, denoted by $td(G)$, is the minimum height of a rooted forest $F$ such that $G\subseteq clos(F)$.

The following result is obvious, which will be frequently used in this paper without reference.

\begin{observation}
Let $F$ be a rooted tree. 
For any connected graph $H$ with $H\subseteq clos(F)$, there is a unique vertex $a\in V(H)$ satisfying $V(H)\subseteq V(F_a)$.
\end{observation}

\begin{lemma}\label{lem:rooted-balance}
Let $F$ be a rooted tree with $S\subseteq V(F)$. For any integer $\ell\geq0$, there exists a set $X\subseteq V(F)$ such that
\begin{enumerate}[label=\textup{(\roman*)}]
  \item $ |X|\le\frac{|S|}{\ell+1}$;
  \item every component of $F\backslash X$ contains at most $\ell$ vertices of $S$; and
  \item $|S\cap V(F_x)|\ge \ell+1$ for every $x\in X$.
\end{enumerate}
\end{lemma}

\begin{proof}
We proceed by induction on $|S|$. When $|S|\leq \ell$, set $X:=\emptyset$ and (i)-(iii) hold. Now suppose that $|S|\geq \ell+1$. Choose a vertex $v$ of greatest depth in $F$ such that $|S\cap V(F_v)|\geq \ell+1$.
By the choice of $v$, every child $w$ of $v$ satisfies $|S\cap V(F_w)|\leq \ell$.
When $v$ is the root of $F$, set $X:=\{v\}$ and (i)-(iii) hold. 
So we may assume that $v$ is not the root of $F$. Set $F':=F\backslash V(F_v)$ and $S':=S\cap V(F')$. Then $|S'|\le |S|-(\ell+1)$ as  $|S\cap V(F_v)|\ge \ell+1$.
By induction, there is a set $X'\subseteq V(F')$ such that (i)-(iii) are true for $F', X'$ and $S'$. Set $X:=X'\cup\{v\}$.
Then (ii) and (iii) obviously hold, and $|X| = |X'|+1 \le \frac{|S'|}{\ell+1}+1 \le \frac{|S|-(\ell+1)}{\ell+1}+1 = \frac{|S|}{\ell+1}$. This proves the lemma.
\end{proof}

\section{Proof of Theorem~\ref{thm:main}}\label{main}
In this section, we prove Theorem~\ref{thm:main}. For convenience, we restate it here.
\setcounter{theorem}{2}
\renewcommand{\thetheorem}{1.\arabic{theorem}}

\begin{theorem}\label{prop:large-h}
For any connected graph $G$ with tree-depth $h$, there is a tree $T$ with $\operatorname{rad}(T)\leq h-1$ and a
$T$-partition of $G$ with width at most $\mathrm{max}\{1, (4h-10)\Delta(G)+1\}$.
\end{theorem}

\begin{proof}
Set $\Delta:=\Delta(G)$. 
When $h\leq 2$, since $G$ is a star or an isolated vertex, $tpw(G)=1$ and the theorem obviously holds. So we may assume that $h\geq3$.
When $|V(G)|\leq (4h-10)\Delta+1$, the theorem also obviously holds. So we may assume that $|V(G)|> (4h-10)\Delta+1$.


Let $F$ be a rooted tree of height $h$ with $V(F)=V(G)$ and $G\subseteq\operatorname{clos}(F)$. Let $\ell\geq0$ be an integer.
We will recursively construct a 
$T$-partition $(B_t: t\in V(T))$ of $G$ for some tree $T$ such that when $|B_t|> (4h-10)\Delta+1$, the following hold.
   \begin{enumerate}[label=\textup{(\roman*)}]
   \item $G[B_t]$ is connected,
  \item $d_T(t)\leq 1$, and 
  \item $|N_G(B_t)|\le h-2+\ell$.
  \end{enumerate}
\noindent 
Let $ L$ be the set consisting of the vertices of $T$ whose indexed bags have size larger than $(4h-10)\Delta+1$. Evidently, when $ L=\emptyset$, we get a tree-partition of $G$ of width at most $(4h-10)\Delta+1$. Hence, it suffices to construct a $T$-partition  of $G$ with $ L=\emptyset$ and $\mathrm{rad}(T)\leq h-1$.

Initially, let $T$ be a tree with $r$ as the unique vertex and set $B_{r}:=V(G)$. 
Since $G$ is connected, (i)--(iii) holds for the initial tree-partition of $G$ and $L=\{r\}$.

When $ L\neq\emptyset$, arbitrarily choose $t\in L$. Since $G[B_t]$ is connected by (i), there is a unique vertex $a_t\in B_t$ with $B_t\subseteq V(F_{a_t})$. 
Set $$S_t:=\{u\in B_t:N_G(u)-B_t\neq\emptyset\}. \eqno{(3.1)}$$ 
Applying Lemma~\ref{lem:rooted-balance} to $F_{a_t}$, $S_t$, and $\ell$, there is a set $X_t\subseteq V(F_{a_t})$ with
\[ |X_t|\leq\frac{|S_t|}{\ell+1},\ \ |S_t\cap V(F_x)| \ge \ell+1,\ \ |S_t\cap V(C)| \le \ell, \eqno{(3.2)}\]
for every $x\in X_t$ and each component $C$ of $F_{a_t}\backslash X_t$. Evidently, when $t=r$, the vertex $a_t$ is the root of $F$ and $S_t=X_t=\emptyset$. Since $|S_t\cap V(F_x)| \ge \ell+1$ for every $x\in X_t$ and $\ell\geq0$,

\ \ \  \  {\bf (a)} for each $x\in X_t$, either $x\in S_t$ and $\ell=0$ or $x$ is not a leaf of $F$.

\noindent Set $$\widehat X_t:=\bigcup_{x\in X_t}V(xFa_t),\ \ B'_t:=\{a_t\}\cup S_t\cup(B_t\cap\widehat X_t), \eqno{(3.3)}$$
where $xFa_t$ denotes the unique $(x, a_t)$-path in $F$.
Since $\{a_t\}\cup S_t\subseteq B_t$ by (3.1), we have $B'_t\subseteq B_t$. Note that $a_t\in B_t\cap\widehat X_t$ when $X_t\neq\emptyset$; and $ B'_t=\{a_t\}$ when $t=r$. 

\begin{claim}\label{size}
When $t\neq r$, we have $|B'_t|\le (4h-10)\Delta+1$ if we choose $\ell=h-3$.
\end{claim}
\begin{proof}
On one hand, for any $x\in X_t$, by (a), we have $|V(xFa_t)-\{a_t\}-S_t|\leq h-2$. Moreover, since $a_t\in B_t\cap\widehat X_t$ when $X_t\neq\emptyset$, we have 
$$|B'_t|\le 1+|S_t|+(h-2)|X_t|\le 1+\left(1+\frac{h-2}{\ell+1}\right)|S_t|, \eqno{(3.4)}$$
by (3.2) and (3.3). On the other hand, since $|N_G(B_t)|\le h-2+\ell$ by (iii), it follows from (3.1) that
$$|S_t|\leq\sum_{v\in N_G(B_t)}d_G(v)\leq\Delta|N_G(B_t)|\leq(h-2+\ell)\Delta. \eqno{(3.5)}$$Combining (3.4) and (3.5), we have
$|B'_t|\leq1+f_h(\ell)\Delta$, where $f_h(\ell)=(h-2+\ell)\left(1+\frac{h-2}{\ell+1}\right)$.
Since $f_h(\ell)-(4h-10)=\frac{(\ell-h+4)(\ell-h+3)}{\ell+1}$ and $\ell-h+4$ and $\ell-h+3$ are consecutive integers, $f_h(\ell)\geq 4h-10$ and the equality holds when $\ell\in\{h-3,h-4\}$. Hence, $|B'_t|\le (4h-10)\Delta+1$ when $\ell=h-3$.
\end{proof}

By Claim \ref{size}, now we fix $\ell=h-3$.
Let $C_1,\dots,C_m$ be the components of $G[B_t-B'_t]$.
\begin{claim}\label{neighbour}
For any $1\leq i\leq m$, we have $N_G(V(C_i))\subseteq B'_t$.
\end{claim}
\begin{proof}
Since $S_t\subseteq B'_t$ implies $V(C_i)\cap S_t=\emptyset$, we have $N_G(V(C_i))\subseteq B'_t$ by (3.1) and the fact $B'_t\subseteq B_t$. 
\end{proof}

Assume that $|V(C_i)|\geq (4h-10)\Delta+1$ for some $1\leq i\leq m$.
Let $a_i$ be the vertex in $V(C_i)$ having minimum depth in $F$. Then $V(C_i)\subseteq V(F_{a_i})$ and the following claim holds.
\begin{claim}\label{clm:active-boundary}
The vertex $a_t$ is a strict ancestor of $a_i$ in $F$ and $|N_G(V(C_i))|\le h-2+\ell$.
\end{claim}
\begin{proof}
Since $C_i$ is a component of $G[B_t-B'_t]$ and $a_t\in B'_t$, the vertex $a_t$ is a strict ancestor of $a_i$ in $F$ by the definitions of $a_t$ and $a_i$. On one hand, since $|V(C_i)|\geq (4h-10)\Delta+1$ implies that $a_i$ is not a leaf of $F$, we have $|\operatorname{Anc}_F(a_i)|\leq h-2$. 
On the other hand, since every vertex in $N_G(V(C_i))$ is either a strict ancestor of $a_i$ or belongs to $V(F_{a_i})$, combined with Claim \ref{neighbour}, we have $N_G(V(C_i))\subseteq \operatorname{Anc}_F(a_i)\cup(B'_t\cap V(F_{a_i}))$. Moreover, since $a_t$ is a strict ancestor of $a_i$ in $F$ and $F_{a_i}$ is contained in a component $C$ of $F_{a_t}\backslash{X_t}$, it follows from the fact  $B'_t=\{a_t\}\cup S_t\cup(B_t\cap\widehat X_t)$ that $ B'_t\cap V(F_{a_i})=S_t\cap V(F_{a_i})\subseteq S_t\cap V(C)$, so $|N_G(V(C_i))| \le |\operatorname{Anc}_F(a_i)|+|B'_t\cap V(F_{a_i})|\le h-2+\ell$ as $|S_t\cap V(C)|\le\ell$ by (3.2).
\end{proof}

Now we update $T$ to a new tree, which is still denoted by $T$, by adding a new vertex $t_i$ adjacent only to  $t$ for each $1\leq i\leq m$. Set $B_{t_i}:=V(C_i)$ and update the bag indexed by $t$ with $B'_t$. Then we get a new $T$-partition of $G$ by Claim \ref{neighbour} and the fact that $(B'_t, B_{t_1},\ldots, B_{t_m})$ is a partition of $B_t$. Moreover, since all $C_i$ are connected, this new tree-partition satisfies (i)-(iii) by Claims \ref{size} and \ref{clm:active-boundary}. Since $a_t\in B'_t\subseteq B_t$, each $V(C_i)$ is a proper subset of $B_t$, so the process will be finished and we will get a $T$-partition of $G$ with width at most $(4h-10)\Delta+1$.  Moreover, since $a_t$ is a strict ancestor of $a_i$ in $F$ by Claim \ref{clm:active-boundary}, $d_T(r,s)\leq h-1$ for any $s\in V(T)$, so $\mathrm{rad}(T)\leq h-1$.
\end{proof}



We now show that the dependence on $\Delta(G)$ in Theorem~\ref{thm:main} is necessary, even for graphs of tree-depth $3$.
\setcounter{theorem}{0}
\renewcommand{\thetheorem}{3.\arabic{theorem}}
\begin{lemma}\label{thm:degree-necessary}
For every integer $n\ge 1$, there exists a graph $G_n$ with $td(G_n)=3$ and with $tpw(G_n)\ge \frac{n}{2}$.
\end{lemma}

\begin{proof}
Let $G_n$ be the graph with $V(G_n)=\{c\}\cup\{a_i, b_{i,j}: 1\le i, j\le n\}$,
with $E(G_n)=\{ca_i, cb_{i,j}, a_ib_{i,j}:\ 1\le i, j\le n\}$.
Let $F_n$ be the rooted tree of height $3$ with $V(F_n)=V(G_n)$ and $E(F_n)=\{ca_i, a_ib_{i,j}:\ 1\le i, j\le n\}$, and with $c$ as its root.
Then $G_n\subseteq clos(F_n)$, so $td(G_n)\le 3$. Moreover, since $G_n$ contains a triangle, $td(G_n)\ge 3$, so $td(G_n)=3$.

We claim that  $tpw(G_n)\ge \frac{n}{2}$. Assume to the contrary that there is a $T$-partition $(B_x:x\in V(T))$ of $G_n$ of width $\omega<\frac{n}{2}$ for some tree $T$.
Let $r\in V(T)$ be the unique vertex with $c\in B_r$. Since $|B_r|\le \omega<\frac{n}{2}$, some $a_i$, say $a_1$, is not in $B_r$. Assume that $a_1\in B_s$ for some vertex $s\in V(T)-\{r\}$. Then $rs\in E(T)$. Moreover, since $b_{1,j}$ is adjacent to both $c$ and $a_1$ for each $1\leq j\leq n$, either $b_{1,j}\in B_r$ or $b_{1,j}\in B_{s}$. So $n+2\le |B_r|+|B_s|\le 2\omega<n$ as $\omega<\frac{n}{2}$, 
a contradiction.
\end{proof}

\section{Acknowledgments}
This research is supported by National Natural Science Foundation of China (12671400), Natural Science Foundation of Fujian Province (2025J01486).

\section{Declaration}
\vspace{5pt}
\noindent$\textbf{Conflict~of~interest.}$
The authors declare that they have no known competing financial interests or personal relationships that could have appeared to influence the work reported in this paper.

\vspace{5pt}
\noindent$\textbf{Data availability.}$
Data sharing is not applicable to this paper as no datasets were generated or analysed during the current study.

\noindent$\textbf{AI Declaration.}$
No AI tools have been used in any way to generate the mathematical ideas of this paper.

\end{document}